\documentclass[12pt,leqno]{amsart}

\usepackage{amsmath}

\usepackage{amssymb}

\usepackage{latexsym}

\usepackage{amsthm}

\begin{document}

\begin{center}
\Large{\bf M\"untz linear transforms of Brownian motion}
\end{center}

\vspace{0.5cm}

\begin{center}
{\bf Larbi Alili$^{(1)}$} \hspace{5mm} {\bf Ching-Tang Wu$^{(2)}$}
\end{center}

\vspace{5mm}

\begin{center}
{\small {\sc ABSTRACT}} \\
\vspace{4mm}
\begin{minipage}{14cm}
We consider a class of Volterra linear  transforms of Brownian motion associated to a sequence of M\"untz Gaussian spaces and determine
explicitly their kernels; the kernels take a simple form when expressed in terms of M\"untz-Legendre polynomials. These are new explicit examples of progressive Gaussian enlargement of a Brownian filtration.  We  give
a necessary and sufficient condition for the existence of
kernels of infinite order associated to an infinite dimensional M\"untz Gaussian space; we also examine when the transformed Brownian motion remains a semimartingale in the filtration of the original process. This completes some partial answers obtained in (\cite{Erraoui-Ouknine-2008}, \cite{hhm-00},
\cite{hm-2004}) to the aforementioned problems in the infinite dimensional case.
\end{minipage}
\end{center}

\vspace{0.5cm}

\vspace{0.5cm}

\theoremstyle{plain}
   \newtheorem{Theorem}{Theorem}[section]
   \newtheorem{Corollary}{Corollary}[section]
   \newtheorem{Lemma}{Lemma}[section]
   \newtheorem{Proposition}{Proposition}[section]

\theoremstyle{definition}
   \newtheorem{Definition}{Definition}[section]
   \newtheorem{Example}{Example}[section]
   \newtheorem{Remark}{Remark}[section]
   \newtheorem{example}{Example}[section]

\def\real{\mathbb{R}}
\def\rational{\mathbb{N}}
\def\measure{\mathbb{P}}
\def\preal{{\mathbb{R}}^+}
\def\ii{{I\hspace{-0.5mm}I}}
\def\iii{{I\hspace{-0.5mm}I\hspace{-0.5mm}I}}
\def\sign{\mbox{\rm{sign}}}
\newcommand\spn{\mbox{\rm{span}}}
\def\dim{\mbox{\rm{dim}}}
\def\span{\mbox{\rm{span}}}
\def\dim{\mbox{\rm{dim}}}


\section{Introduction}
There has been a renewed interest in M\"untz spaces which is particularly motivated by topics related to Markov inequalities and approximation theory, see  for example (\cite{BBR-04}--\cite{Borowein-Erdelyi-Zhang-94}) and the references therein. In the meanwhile, Volterra transforms with non square-integrable kernels, involving some functional spaces, provide interesting examples of noncanonical decompositions of the Brownian filtration. This motivated many studies on the topic, for instance see (\cite{Baudoin}, \cite{Chiu-1995}, \cite{Erraoui-Ouknine-2008}, \cite{Hibino-96},   \cite{Levy-56}). Our aim in this paper is to study the class of Volterra transforms involving Gaussian spaces which are generated from sequences of M\"untz polynomials. This gives new explicit examples of progressive enlargement of filtrations and interesting links with M\"untz-Legendre polynomials;  see  (\cite{jeulin80}--\cite{jeulin-yor93}, \cite{yor92}) for studies on this topic in more general frameworks.

To be more precise, let us fix our mathematical setting. Let $B:=(B_t, t\geq 0)$ be a standard Brownian motion defined on a
complete probability space $(\Omega, {\mathcal F}, \mathbb{P})$, and
denote by $\{\mathcal{ F}_t^B, t\geq 0\}$ the filtration it
generates. We encountered, in literature, two types of linear transforms of $B$ which are of our interest in this paper.  The first type consists of transforms of the form
\begin{equation}\label{Volterra-general}
T(B)_t=\int_0^t\! \rho(t/s)\, dB_s
\end{equation}
for all $t>0$ and some  $\rho\in \mathcal{M}$, with
\begin{equation*}
\mathcal{M}=\{\rho:[1,\infty)\rightarrow \mathbb{R}\; \hbox{measurable function s.t.}\; \int_0^1\! \rho^2(1/v)\, dv<\infty \}.
\end{equation*}
These transforms were intensively studied in \cite{jeulin-yor93}. In particular, we found in Proposition 15  therein a variant of Theorem 6.5 of \cite{Knight} which states that $(T(B)_t, t\geq 0)$ is a semimartingale relative to the filtration of $B$ if and only if there exists $c\in \mathbb{R} $ and $g\in \mathcal{M}$ such that
\[
\rho(.)=c+\int_1^{.}\! \frac{1}{y}g(y)\, dy.
\]
 The second type consists of Volterra transforms with non-square-integrable kernels which are of the form
  \begin{equation}\label{Volterra}
T(B)_t=B_t-\int_0^t\! ds\int_0^s\! l(s,v)\, dB_v
\end{equation}
for all $t>0$, where the kernel $l: \mathbb{R}_+^2 \rightarrow \mathbb{R}$, which satisfies $l(s,v)=0$ for $s<v$, is such that the symmetrized kernel
 \begin{eqnarray*}
 \tilde{l}(t,s):=
 \left\{
\begin{array}{lll}
    l(t,s) &\hbox{if}\quad s\leq t;\\
  l(s,t)  &\hbox{if}\quad s\geq t,
\end{array}
\right.
\end{eqnarray*}
 is continuous on $\mathbb{R}^2_+$. These transforms were studied for example in (\cite{alili-wu-01}, \cite{fwy99}, \cite{hhm}). Note that,  in the semimartingale case with $c=1$, the transform (\ref{Volterra-general}) becomes a Volterra transform of the form (\ref{Volterra}) with kernel $l(t,s)=t^{-1}g(t/s)$ for $s\leq t$ and $l(t,s)=0$ otherwise. Conversely, all transforms of the form (\ref{Volterra}) which we will consider in this paper are of the form (\ref{Volterra-general}). Uniqueness when defining  $T(B)$ by either (\ref{Volterra-general}) or (\ref{Volterra}) holds only up to a stochastic modification and we work with the continuous one.

Let $f_j(x):=x^{\lambda_j}$, $j=1,2,\cdots$, be a sequence of M\"untz polynomials where $\Lambda=\{\lambda_1, \lambda_2, \cdots \}$ is a sequence of reals satisfying
\begin{equation}\label{condition-1}
\lambda_j>-1/2, \qquad j=1,2,\cdots.
\end{equation}
These generalized polynomials are defined on $[0, \infty)$ and the value of $f_j$ at $x=0$ is the limit of $f_j(x)$ as $x\rightarrow 0$  from $(0, \infty)$ for $j=1,2, \cdots$.
 For each fixed  $t>0$, let us define the M\"untz Gaussian spaces
\begin{equation}\label{Muntz-g}
G_{t}(\lambda_1,\cdots \lambda_n; B)=\hbox{Span} \left\{ \int_0^t\! s^{\lambda_j} \, dB_s, j=1,2, \cdots n \right\}
\end{equation}
and
\begin{equation}\label{Muntz-g-infty}
G_{t}(\lambda_1, \lambda_2, \cdots; B)=\hbox{Span} \left\{ \int_0^t\! s^{\lambda_j} \, dB_s, j=1,2, \cdots  \right\}
\end{equation}
and let $H_t(B)$ be the closed linear span of $\{B_s, s\leq t \}$. A  M\"untz transform of order $n$  associated to $\lambda_1$, $\cdots$, $\lambda_n$, is a linear transform $T_n$ of the form (\ref{Volterra-general}) such that the following two properties hold true:
\begin{itemize}
\item[(i)] $(T_n(B)_t, t\geq 0)$ is a Brownian motion;
\item[(ii)] we have the orthogonal decomposition
\begin{equation}\label{orthogonal-Muntz}
H_t(B)=H_t(T_n(B))\oplus G_t(\lambda_1,\cdots, \lambda_n; B), \qquad t>0.
\end{equation}
\end{itemize}
Following \cite{alili-wu-01} and \cite{hhm}, if $n<\infty$ then such transforms exist. As we shall see, the transform $T_n$ of the form (\ref{Volterra}) with  $l(t,s)=k_n(t,s) :=t^{-1}K_n(s/t)$ for $s\leq t$ where $ K_n(x):= \sum_{j = 1}^n
a_{j,n} x^{\lambda_j}$ for $0\leq x\leq 1$ and  $a_{1,n},
a_{2,n},...,a_{n,n}$ are uniquely determined by the system
\begin{equation}\label{system-0}
\sum_{j = 1}^n \frac{a_{j,n}}{\lambda_j + \lambda_k + 1} = 1,\quad k
= 1, \cdots, n,
\end{equation}
is a M\"untz transform of order $n$. The latter system, when $\lambda_j=j$ for $j=1,2, \cdots$, was first discovered by P. L\'evy, see  (\cite{Levy-book}, \cite{Levy-57}) and was further studied  in \cite{Chiu-1995}. Solving (\ref{system-0}), we found that the sequence of M\"untz polynomials
$(K_n, n=1,2,\cdots)$ can be simply expressed in terms of M\"untz-Legendre polynomials which  allows to simplify the study of some of their properties. Note that $T_n$ takes the form (\ref{Volterra-general}) with
\begin{equation}\label{Form-kernel}
\rho_n(x)=1-\int_1^x\! K_n(1/r)\frac{dr}{r}, \quad x\geq 1.
\end{equation}
 Thus, we are in  the semimartingale case with $c=1$ and $g(.)=-K_n(1/.)$.

The kernels described above are homogeneous of degree $-1$ in the sense that
\[k_{n}(\alpha t,\alpha s)={\alpha}^{-1}k_n(t,s), \quad 0<s\leq t<\infty, \alpha > 0.
\]
As a consequence, the associated Volterra transforms have a close
connection to a class of stationary processes. That is, the process
\[
(e^{-u/2} T_n(B)_{e^u},
u\in \mathbb{R})
 \]
 is an Ornstein-Uhlenbeck process; the conventional value  $-1/2$ of the parameter will be dropped from our notations. It has the moving average representation, m.a.r.$\,$for short,
\[
S_n(W)_u:=\int_{-\infty}^{u}\!\eta_{n}(u-r) \,
dW_{r}
\]
for $u\in\mathbb{R}$, where $W$ is a Brownian motion indexed by $\mathbb{R}$ and
$\eta_{n}\in L^1(\mathbb{R}_+)\cap L^2(\mathbb{R}_+)$ has the Fourier transform
\begin{equation}\label{fip}
\hat{\eta}_n(\xi):=\int e^{i\xi x}\eta_n(x)\, dx=\frac{1}{1/2-i\xi}\prod_{j=1}^{n}
\frac{\xi-ip_j}{\xi+ip_j}, \qquad \xi \in \mathbb{R},
\end{equation}
where $p_j=\lambda_j+\frac{1}{2}$ for $j=1, 2, \cdots$. Applying then the characterization given in  \cite{Kara-50}, the presence of the inner part, given by the product in
 (\ref{fip}), implies that the latter m.a.r.$\,$is not
canonical with respect to $W$.


A natural question is to know whether there exists a transform $T$ of the form (\ref{Volterra-general}) such that
\[
H_t(B)=H_t(T(B))\oplus G_t(\lambda_1,\lambda_2,\cdots; B)
\]
 for all $t>0$. Partial answers are given in (\cite{Erraoui-Ouknine-2008}, \cite{hhm-00},
\cite{hm-2004}) where the authors established the existence of such transforms. In
particular, for an infinite sequence $\Lambda$ satisfying either $\sup
\lambda_j=+\infty$ or $0<\lambda_1<\lambda_2<...$ there exists no such a transform such that $(T(B)_t, t\geq 0)$ is a semimartingale relative to the filtration of $B$. We see that a necessary and sufficient condition for the existence of  transforms of the form (\ref{Volterra-general}) with infinite dimensional orthogonal complement is
\begin{equation}\label{MS}
\sum_{j=1}^{\infty}\frac{p_j}{p_j^2+1}<\infty.
\end{equation}
This is the well-known M\"untz-Szasz condition which is  necessary and sufficient  for $f_1$, $f_2$, $\cdots$, to be incomplete in $L^2[0,1]$, see for example \cite{Borwein-Erdelyi-95}.
Furthermore,  $(T(B)_t, t\geq 0)$ is a semimartingale relative to the filtration of $B$ if and only if $(\lambda_k)$ is bounded and satisfies (\ref{MS}). Plainly, the latter happens if and only if  $\sum_{j=1}^{\infty} p_j<\infty$.

\section{{\bf M\"untz Gaussian spaces and transforms}}\label{Muntz}
Throughout this paper, we assume that $\Lambda=\{\lambda_1,\lambda_2,\cdots\}$ is a sequence of distinct real numbers satisfying condition (\ref{condition-1}). Thus, the generalized M\"untz
polynomials $f_j(x):=x^{\lambda_j}$, for $j=1,2,\cdots$,  lie in
\[
L_{loc}^2(\mathbb{R}_+):=\{f:\mathbb{R}_+\rightarrow \mathbb{R}; f\in L^2[0,t]\: \hbox{for all}\: 0<t<\infty \}.
\]
 For $t>0$, let us  introduce
\begin{equation*}
M_{n,t}=\hbox{Span}\{ x^{\lambda_1}, x^{\lambda_2}, \cdots, x^{\lambda_n}; x\in [0,t]\}
\end{equation*}
and
\begin{equation*}
 M_{\infty, t}=\hbox{Span}\{ x^{\lambda_1}, x^{\lambda_2}, \cdots; x\in [0,t]\}
\end{equation*}
which are called M\"untz spaces. An associated orthogonal system, known as M\"untz-Legendre polynomials, is specified by $L_1(x)=x^{\lambda_1}$ and $L_2$, $L_3$ $\cdots$, described by
\begin{equation}
\label{eqn:M-K18} L_k(x) = \sum_{j = 1}^k c_{j,k}
x^{\lambda_j},\quad c_{j,k} = \frac{ \prod_{l = 1}^{k-1}
(\lambda_l + \lambda_j + 1)}{ \prod_{l = 1, l \not= j}^{k}
(\lambda_j - \lambda_l)},\quad k=2,3,\cdots,
\end{equation}
see
  \cite{Borwein-Erdelyi-95} and
\cite{Borowein-Erdelyi-Zhang-94}; note that we use slightly different notations since we start the sequence $\Lambda$ with $\lambda_1$ instead of $\lambda_0$.  Recall that $L_k(1)=1$ for $k=1, 2, \cdots$.   Next, to the
linear spaces $M_{n,t}$ and $M_{\infty,t}$ we associate,
respectively, the families of M\"untz Gaussian spaces defined by (\ref{Muntz-g}) and (\ref{Muntz-g-infty}). Recall that the closed linear span of $\{B_s, s\leq t \}$, or the first Wiener chaos of $B$, is given by
\begin{equation}\label{FWC}
H_t(B) = \left\{ \int_0^t\! f(u) \,
dB_u; f \in L^2[0,t]\right\}.
\end{equation}
It follows that  the orthogonal
complements of   $G_{t}(\lambda_1, \cdots, \lambda_n;B)$ and
$G_{t}(\lambda_1, \lambda_2\cdots; B)$, in $H_t(B)$, are respectively given by
\begin{equation*}
 G_{t}^{\perp}(\lambda_1, \cdots, \lambda_n; B) = \left\{ \int_0^t\! f(u) \, dB_u; f \in
L^2[0,t], \int_0^t\! f(s)p(s) \, ds = 0, p \in M_{n,t}
\right\}
\end{equation*}
and
\begin{equation*} G^{\perp}_{t}(\lambda_1, \lambda_2, \cdots; B) = \left\{ \int_0^t\! f(u)\, dB_u; f \in
L^2[0,t], \int_0^t\! f(s)p(s)\, ds = 0, p \in M_{\infty, t}
\right\}.
\end{equation*}
Following \cite{alili-wu-01}, the transform $T$ (resp.$\:$kernel) defined by (\ref{Volterra}) is called a Goursat-Volterra transform (resp.$\:$kernel) of order $n$ if $(T(B)_t, t\geq 0)$ is a Brownian motion and there exists $n$ linearly independent functions $g_j\in L^2_{loc}(\mathbb{R}_+)$ such that
\[
H_t(B)=H_t(T(B))\oplus \hbox{Span}\{\int_0^t\! g_j(s)\, dB_s, j=1,2,\cdots n \}
\]
for all $t>0$. We are now ready to determine a Goursat-Volterra transform $T_n$ associated to the M\"untz polynomials $f_1$, $f_2$, $\cdots$, $f_n$, in the case when $n$ is finite.

\begin{Theorem}
\label{mg-thm}  Assuming that $n<\infty$ then
\begin{eqnarray*}
 k_n(t,s):=
 \left\{
\begin{array}{lll}
    t^{-1}K_n(s/t) &\hbox{if}\quad s\leq t;\\
  0  &\quad \hbox{otherwise},
\end{array}
\right.
\end{eqnarray*}
where
\begin{equation}
\label{eqn:M-K33-2}K_n(s)=\sum_{j =
 1}^n a_{j,n}
 s^{\lambda_j},\quad
a_{j,n} = \frac{\prod_{l = 1}^n (\lambda_j + \lambda_l +
1)}{\prod_{l=1, l\neq j}^{n}(\lambda_j - \lambda_l)}, \quad j=1,...,
n,
\end{equation}
is a Goursat-Volterra kernel of order $n$. Furthermore, writing $T_n$ for the Goursat-Volterra transform
associated to $k_n$, the orthogonal complement of $H_{t}(T_n(B))$ in $H_t(B)$ is
$G_{t}(\lambda_1, \cdots, \lambda_n; B)$ for all $t\geq 0$. Note that $T_n$ is of the form (\ref{Volterra-general}) with $\rho$ prescribed by (\ref{Form-kernel}).
\end{Theorem}

\begin{proof}[Proof of Theorem \ref{mg-thm}] $T_n(B)$ is a Brownian motion  if and only if $k_n$ satisfies the self-reproduction property
\begin{equation}
\label{eqn:K10} k_n(t,s) = \int_0^s\! k_n(t,u) k_n(s,u)\, du
\end{equation}
for a.e. $s \le t$, which is found in Theorem 6.1 in \cite{fwy99}. This is obtained by writing $\mathbb{E}[T_n(B)_tT_n(B)_s]=s\wedge t$, differentiating and rearranging terms.
 But, if we set $k_n(t,s) = t^{-1} \sum_{j = 1}^n
a_{j,n} ( s/t)^{\lambda_j}$, then (\ref{eqn:K10}) is equivalent to saying that $(a_{j,n},
j=1,2,\cdots, n)$ solves the linear system (\ref{system-0}). To study the system, consider the
$n$-degree polynomial
\[
p_n(x) = \prod_{j = 1}^n (x + \lambda_j + 1)
- \sum_{k = 1}^n a_{k,n} \prod_{j = 1, j \not= k}^n (x + \lambda_j +
1)\]
which, of course, has at most $n$ roots. But $p_n(x) = 0$ is
equivalent to
$
\sum_{k = 1}^n \frac{a_{k,n}}{x + \lambda_k + 1} =
1$. This fact, when combined with $\lim_{x\rightarrow
\infty}p_n(x)/x^n=1$, implies that $p_n(x) = \prod_{j = 1}^n (x -
\lambda_j)$. Now, let us choose $m \in \{ 1, \cdots, n \}$ and substitute
the latter product formula in the expression of $p_n(x)$. Dividing
 both sides by $\prod_{j\not=m} (x + \lambda_j + 1)$,
rearranging terms and letting $x\rightarrow -(\lambda_m + 1)$ we
obtain the expressions of $a_{1,n}$,$
a_{2,n}$, $\cdots,$  $a_{n,n}$. Next, $k_n$ is a Volterra kernel because
it is continuous on $\{(u,v)\in \mathbb{R}_+^2;  u>
v \}$ and satisfies the following integrability condition which is enough for (\ref{Volterra}) to be well defined. We have
\begin{eqnarray*}
\int_0^t\! \left(\int_0^u k_n^2(u,v) \, dv \right)^{1/2}\, du&=&\int_0^t\!
\left(\int_0^1 k_n^2(u,u r) u \, dr \right)^{1/2}\, du\\
&=&2 \sqrt{t}\left(\int_0^1\! K_n^2(r) \, dr \right)^{1/2}\\
&=&2 \sqrt{t}K_n^{1/2}(1)<+\infty,
\end{eqnarray*}
where we used the homogeneity and the self-reproduction properties of
$k_n$.  Finally, we need to identify $H_t(B)\ominus H_t(T_n(B))$ for an arbitrarily fixed $t>0$. The condition $\int_0^t\! f(u)\, dB_u\, \bot\, T(B)_s$ for all $s\leq t$ is equivalent to
\begin{eqnarray}\label{int-equation}
\int_0^s\, f(r)\, dr = \int_0^s\! du\int_0^u k_n(u,v) f(v)\, dv.
\end{eqnarray}
If we write $k_n(u,v)=\sum_{j=1}^n \varphi_j(u) f_j(v)$ then by differentiating the latter equation we obtain the integral equation
\begin{eqnarray*}
f(s)  &=&\int_0^s\! k_n(s,v) f(v)\, dv\\
&=& \sum_{j=1}^n  \varphi_j(s)\int_0^s\, f_j(v) f(v)\, dv
\end{eqnarray*}
for a.e. $t>0$, this can also be found in \cite{jeulin-Yor-90}. Clearly, if $f$ solves it then $f(t)/\varphi_1(t)$ must be absolutely continuous with respect to the Lebesgue measure. Repeating this argument, we see that (\ref{int-equation}) is  equivalent
to an ordinary linear differential equation of degree $n$ which should hold for a.e. $s\in [0,t]$. The
functions $u\rightarrow u^{\lambda_j}$, $j = 1, \cdots, n$, being $n$ linearly
independent solutions, we conclude that
$G_{t}(\lambda_1, \cdots, \lambda_n; B)$ is the orthogonal complement of $H_{t}(T_n(B))$ in $H_t(B)$ as required.  Next, by using the homogeneity property of  the kernel $k_n$ and the stochastic Fubini theorem, we can
write
\begin{eqnarray*}
T_n(B)_t&=&B_t - \int_0^{t}\! \int_0^u\! k_n(u,v)\,   dB_v\, du \\
&=& \int_0^t\! \left(1-\int_v^t k_n(u,v)\, du \right)  \,dB_v \\
&=& \int_0^t\! \left(1-\int_{1}^{t/v} k_n(vr,v)v\,  dr \right)\,  dB_v\\
 &=& \int_0^t\! \rho_n(t/v) \, dB_v, \quad t\geq 0,
\end{eqnarray*}
where
\begin{equation*}
\rho_n(x)=1-\int_{1}^{x}\! k_n(r, 1) \, dr=1-\int_1^x\! K_n(1/r)\frac{dr}{r}.
\end{equation*}
\end{proof}
\begin{Remark} Since when $n<\infty$ we have $K_n\in L^2[0,1]$, by using the homogeneity property of $k_n$ we obtain that
$\int_0^1\! ds\int_0^1\! k_n^2(s,v)\, dv=+\infty$ i.e.
$k_n\notin L^2([0,1]\times[0,1])$.  The representation (\ref{Volterra})  is the  semimartingale decomposition of $T_n(B)$ with respect to $\mathcal{F}^{B}$; it is noncanonical relative to the filtration of $B$ since $\mathcal{F}_t^{T(B)}\varsubsetneq \mathcal{F}^{B}_t$ for all $t>0$. The Volterra representation of $T_n(B)$ as
 $
T_n(B)_t=X_t-\int_0^t\! ds\int_0^s\! l(s,v)\, dX_v$,
where $l:\mathbb{R}_+^2\rightarrow \mathbb{R}$ is such that $l(s,v)=0$ for $s<v$ and $X$ is a Brownian motion, is not unique. Indeed, one representation is given with $X=B$ and $l=k_n$ and another one is given with $X=T_n(B)$
and $l\equiv 0$. But if we add the condition  $l\in L^2([0,1]\times [0,1])$ then the representation above is unique, see \cite{Hitsuda-68}.
\end{Remark}

The covariance matrix
\[
(m_t^n)_{lj}=\frac{t^{\lambda_l+\lambda_j+1}}{\lambda_l+\lambda_j+1}, \quad l,j=1,2,\cdots, n,
 \]
 of the Gaussian process $(\int_0^t\! f^*(s)\, dB_s)$, where $f:=(f_1, \cdots, f_n)^*$ is the transpose of the row vector $(f_1, \cdots, f_n)$, has an inverse matrix which we denote by $\alpha_t^n$. In fact, $m_1^n$ is a Cauchy matrix and
an explicit formula for its inverse can be found  in (\cite{Gohberg-Koltracht}, \cite{Schechter}).  Note also that the Goursat form of $k_n$ given below  is given in a semi-explicit form in  \cite{hhm}.  Here we propose another method to compute the entries of $\alpha_t^n$ and $\varphi$.

\begin{Proposition}\label{Matrix-form} The kernel $k_n$  of Theorem \ref{mg-thm} satisfies
\begin{eqnarray*}
 k_n(t,s)=
 \left\{
\begin{array}{lll}
    \varphi^*(t)\cdot f(s) &\hbox{if}\quad s\leq t;\\
  0  &\hbox{otherwise},
\end{array}
\right.
\end{eqnarray*}
where $\varphi(.)=\alpha(.) \cdot f(.)$,  $\varphi_{l}(t)=a_{l,n}t^{-\lambda_l-1}$, $l=1,2,\cdots, n
$ and the entries of $\alpha_t^n$ are given by $
(\alpha_t^n)_{l,j}= a_{l,n}a_{j,n}(\lambda_l+\lambda_j+1)^{-1} \,
t^{-\lambda_l-\lambda_j-1}$.
\end{Proposition}

\begin{proof} [Proof of Proposition \ref{Matrix-form}] Assume that $k_n$ is of the given Goursat form where $\varphi_1,\varphi_2, \cdots, \varphi_n$ are unknown. By using the self-reproduction property (\ref{eqn:K10}) a little algebra gives that $\varphi(.)=\alpha(.) \cdot f(.)$.   The entries of $\varphi(t)$ are
identified from the expression of $k_n$ given in Theorem
\ref{mg-thm}. Next, from Theorem 2.2 in \cite{alili-wu-01} , we know that
$(\alpha_{t}^n, t>0)$ is given, in terms of $\varphi$, by
\begin{equation*}
\alpha_t^n= \int_t^{\infty}\! \varphi(u) \cdot \varphi^*(u) \, du +
\alpha_{\infty}, \quad \varphi(t) = \alpha_t \cdot f(t), \quad
t>0.
\end{equation*}
But, here $\alpha_{\infty}^n\equiv 0$ because $f_1$, $f_2$,
$\cdots$, $f_n$ are not square-integrable over $(0, +\infty)$. Plugging in the
vector $\varphi$ we obtain the matrix $\alpha_t^n$.
\end{proof}
\begin{Remark} \label{Bridges}
In terms of filtrations, for $n<\infty$ and $0<T<\infty$, we have
${\mathcal F}_T^B = {\mathcal F}_T^{T_n(B)} \otimes \sigma
\left(G_T(\lambda_1, \cdots, \lambda_n;B) \right)$.
In fact, $\mathcal{F}_T^{T_n(B)}$ coincides, up to null sets, with
$\sigma \{B_u^{(br)}, u\leq T \}$, where  $(B_u^{(br)}, u\leq T)$ is
the $f$-generalized bridge over the interval $[0,T]$. A realization of this is given by
$
B^{(br)}_u=B_u-\psi^*_T(u)\cdot \int_0^{T}\! f(s)
\, dB_s$
 where $\psi_T(u)=\alpha_T^n \cdot \int_0^u\! f(r)\, dr$, for $u<T$. This is called a generalized bridge because $\int_0^T\! f_j(s)\, dB_s^{(br)}=0$ for $j=1,\cdots, n$. Note that $T_n(B^{(br)})=T_n(B)$ on $(0,T)$; we refer to
\cite{alili00} for more details on these processes.
\end{Remark}

The objective of the next proposition is to show that we can express $K_n$  in terms of M\"untz-Legendre
polynomials given by formula (\ref{eqn:M-K18}) which form an orthogonal basis of $M_{n,1}$. In the special case when $\lambda_j=j$ for all $j$, an integro-difference equation satisfied by $\rho_n$, $n=1,2, \cdots$, was discovered in \cite{Chiu-1995}. The second assertion of the following result proves useful for finding the analogue of Chiu's result in the general M\"untz framework.

\begin{Proposition} \label{properties} Recall that the functions $L_n$ and $K_n$  are given by (\ref{eqn:M-K18}) and (\ref{eqn:M-K33-2}), respectively.
The following assertions hold true.

1) We have
\begin{equation}\label{expression-Muntz-Legendre}
K_n(x)=
\sum_{j=1}^{n}(1+2\lambda_j)L_{j}(x), \quad x\leq 1.\end{equation}
In particular, $K_n(1)=\sum_{j=1}^n (1+2\lambda_j)$. Consequently, we have
\[
K_n(x)=x^{-\lambda_n}\frac{\partial}{\partial x}\left(x^{\lambda_n
+1}L_n(x) \right)
\]
and, equivalently,
\[
L_n(x)=x^{-\lambda_n-1}\int_0^x\! s^{\lambda_n}K_n(s)\,  ds.
\]
 Note that unlike
M\"untz-Legendre polynomials, the M\"untz polynomial  $K_n$ does not depend on
the order of $\lambda_1, \lambda_2,\cdots, \lambda_n$.

2) The
sequence $K_{n}$, $n=1,2, \cdots$, satisfies the integro-difference equation
\[
K_n(x)=K_{n-1}(x)+(2\lambda_{n}+1)x^{\lambda_n}\left(1-\int_x^1\!
u^{-\lambda_n-1}K_{n-1}(u)\, du\right).
\]
\end{Proposition}

\begin{proof}[Proof of Proposition \ref{properties}] 1) We have $(1+2\lambda_n)c_{n,n}=a_{n,n} $
and $(1+2\lambda_n)c_{j,n}=a_{j,n}-a_{j,n-1} $ for $j=1,2, \cdots, n-1$.
Thus, we can write
\begin{eqnarray*}
K_n(x)-K_{n-1}(x)&=&a_{n,n}x^{\lambda_n}+\sum_{j=1}^{n-1}(a_{j,n}-a_{j,n-1})x^{\lambda_j}\\
&=&(1+2\lambda_n)c_{n,n}x^{\lambda_n}+(1+2\lambda_n)\sum_{j=1}^{n-1}c_{j,n}x^{\lambda_j}\\
&=&(1+2\lambda_n)L_n(x).
\end{eqnarray*}
  Iterating, with the convention that $K_0\equiv 0$, and summing up the equations we get the first formula;  $K_n(1)$ is obtained by setting $x=1$ an using $L_j(1)=1$ for $j=1,2,\cdots, n$. As a by-product formula, we note that
$
(\lambda_j+\lambda_n+1)c_{j,n}=a_{j,n}$ for  $j\leq n$.
 The second assertion is easily obtained by integration.

2) We quote from
\cite{Borowein-Erdelyi-Zhang-94} the recurrence formula
\begin{eqnarray*}
x^{\lambda_n+\lambda_{n-1}+1}\left(x^{-\lambda_n}L_n(x)
\right)'=\left(x^{\lambda_{n-1}+1}L_{n-1}(x) \right)'.
\end{eqnarray*}
Combining this with the first assertion and simplifying yields
\begin{eqnarray*}
\left(x^{-\lambda_n}L_n(x) \right)'&=&x^{-\lambda_n
-1}K_n(x)-(2\lambda_n+1)x^{-2\lambda_n-2}\int_0^x \!
s^{\lambda_n} K_n(s)\, ds.\\
&=&x^{-\lambda_n-1}K_{n-1}(x).
\end{eqnarray*}
Differentiating, we find $-\lambda_n K_n(x) +xK'_n(x)=(\lambda_n
+1)K_{n-1}+xK'_{n-1}$. This is nothing but a differential form of the
integro-difference equation. It remains to use the first assertion on the form
$K_n(x)=K_{n-1}(x)+(1+2\lambda_n)L_n(x)$ and the fact that
$L_n(1)=1$ to conclude.
\end{proof}
Our aim now is to outline a connection between self-reproducing kernels and the classical kernel systems.
\begin{Proposition} \label{reproduction}For each fixed $t>0$, the kernel system associated to $M_{n,t}$
is given by $g_{n,t}(u,v)=\frac{1}{t} \sum_{l=1}^{n} (1+2\lambda_l)L_l(\frac{u}{t})L_l(\frac{v}{t})$ for
$0< u,v\leq t$. Letting  $u\rightarrow t$ we get that
$k_n(t,s)=g_{n,t}(t,s)=\frac{1}{t}\sum_{l=1}^{n} (1+2\lambda_l)L_l(\frac{s}{t})$ for $0<s\leq t<\infty$.
\end{Proposition}
\begin{proof}[Proof of Proposition \ref{reproduction}]
The kernel system is given by $g_{n,t}(u,v)=\sum_{k=1}^{n}q_{k,t}^n(u)q_{k,t}^n(v)$ where
$(q_{k,t}^n, n=1, \cdots, n)$ is an orthonormal sequence that generates  $M_{n,t}$. This is a reproducing kernel in the sense that,  for any $Q_t\in M_{n,t}$, we have
\[
Q_t(u)=\int_0^t\! g_{n,t}(u,v)Q_t(v) \, dv.
  \]
  Exploiting homogeneity, we easily check that the sequence $(q_{j,t}^n(x), x\in [0,t]; j=1,2,\cdots, n)$ defined by $q_{m,t}^n(u):=\sum_{k=1}^{m}c_{k,m}(t)u^{\lambda_{k}}=\sqrt{(1+2\lambda_m)/t}L_{m}(u/t)$ satisfies the requirements. We conclude using continuity and the fact that $L_n(1)=1$.
\end{proof}

\section{Connection to stationary Ornstein-Uhlenbeck processes}
We discuss here a question tackled in \cite{hm-2004}; this  consists
of determining a necessary and sufficient condition for the
existence of transforms of the form (\ref{Volterra-general}) or (\ref{Volterra})  with an infinite dimensional orthogonal complement
associated to
$\Lambda$. Let us recall some excerpts from \cite{jeulin-yor93} and \cite{Knight} on linear transforms of Brownian motions and stationarity. If two semimartingales $W$ and $B$ are related by
  \begin{eqnarray}\label{def-W}
 W_u=
 \left\{
\begin{array}{lll}
    B_1+\int_1^{e^u}\! \frac{dB_s}{\sqrt{s}} &\hbox{if}\quad u\geq 0;\\
    B_1-\int_{e^u}^1\! \frac{dB_s}{\sqrt{s}}& \hbox{if} \quad u\leq 0,
\end{array}
\right.
\end{eqnarray}
  and, equivalently, by
  \begin{equation}\label{def-B}
B_t=\int_{-\infty}^{\log{t}}\! e^{r/2}\, dW_r, \quad t>0,
 \end{equation}
then it is easily checked that $B$ is standard Brownian motion if and only if $W$ is a Brownian motion indexed by $\mathbb{R}$ i.e. $W$ is a  centered continuous Gaussian process with independent increments such that $\mathbb{E}\left [ (W_u-W_v)^2\right]=|u-v|$ for all $u$ and $v\in \mathbb{R}$. Furthermore, we have $\int_0^{\infty}\! \varphi(s)\, dB_s=\int_{\mathbb{R}}\! V\varphi (r)\, dW_r$ for $\varphi\in L^2(\mathbb{R}_+)$
 where the isometry $V:L^2(\mathbb{R}_+)\rightarrow  L^2(\mathbb{R})$ is defined by $V\varphi(u)=e^{u/2}\varphi(e^{u})$. We need to introduce the mapping  $U:C(\mathbb{R}_{+}, \mathbb{R})\rightarrow C(\mathbb{R}, \mathbb{R})$ which is specified by $U\varphi(u)=e^{-u/2}\varphi(e^{u})$, for $u\in \mathbb{R}$, and denote by $U^{-1}$ its inverse operator.
 Keeping in mind that $T$ is defined by (\ref{Volterra-general}) and setting  $\Theta=U\circ T$, we clearly have that  $\Theta(B)=S(W)$  where the transform $S$ is defined by
  \begin{equation}\label{MA}
S(W)_u=\int_{-\infty}^u\! \eta(u-v)\, dW_v, \quad u\in \mathbb{R},
\end{equation}
 with $\eta(u)={{\bf 1}}_{u>0}U\rho(u)$.
 Plainly, $T(B)$ is a Brownian motion if and only if  $\Theta(B)$ is a stationary Ornstein-Uhlenbeck process. Moreover, for some $f\in L^2_{loc}(\mathbb{R}_+)$  we have
 \[
 H_t(T(B))\perp \int_0^t\! f(s)\, dB_s, \quad t>0,
 \]
 if and only if
 \[
 \mathcal{H}(S(W))_u\perp \int_{-\infty}^u\!  V(f)(r)\, dW_r,\quad u\in \mathbb{R},
 \]
  where $\mathcal{H}(S(W))_u$ stand for the closed linear span of $\{S(W)_r, r\leq u\}$. The focus in the next result is on the m.a.r.$\,$of $\Theta_n(B):=(U\circ T_n(B)_u,
u\in\mathbb{R})$ in case when $n<\infty$.
\begin{Proposition}\label{moving-average-representation}
 Assume that $n<\infty$ and $T_n$ is the Goursat-Volterra transform of Theorem \ref{mg-thm}. The process $\Theta_n(B)$ has the m.a.r.$\,$(\ref{MA})
where $W$ is given by (\ref{def-W}) and $\eta:=\eta_{n}$ has the Fourier transform given by
\begin{equation}\label{product}
\hat{\eta}_n(\xi)=(1/2-i\xi)^{-1}\Pi_n(\xi), \quad \xi\in \mathbb{R},
\end{equation}
where
\begin{equation*}
 \Pi_n(\xi):=\prod_{j=1}^{n}
\frac{\xi-ip_j}{\xi+ip_j}
\end{equation*}
and $p_j=\frac12 +\lambda_j$ for $j=1, \cdots, n$.
\end{Proposition}
\begin{proof}[Proof of Proposition \ref{moving-average-representation}]
Using Theorem \ref{mg-thm} and the recalls above, we see that formula (\ref{MA}) holds  with $\eta_n(t)={\bf 1}_{\{ t > 0 \}} U \circ \rho_n(t)$ and $\eta=\eta_n$. Note that $\eta_n\in L^1(\mathbb{R}_+)\cap L^2(\mathbb{R}_+)$. Now, for $\xi \in \mathbb{R}$,  we have
\begin{eqnarray*}
\hat{\eta}_n(\xi)&=&\int_0^{\infty}\! e^{i\xi t} e^{-t/2}\rho_n(e^t)\, dt\\
&=& \int_0^{\infty}\! e^{-(1/2-i\xi)t} \left(1-\int_1^{e^t}\! K_n(1/r)(1/r)\, dr \right)\, dt\\
&=& (1/2-i\xi )^{-1}-\int_{1}^{\infty}\! \left(\int_{\ln r}^{\infty}\! e^{-(1/2-i\xi)t} dt\right)K_n(1/r)(1/r)\, dr\\
&=&\frac{1}{1/2-i\xi} \left(1-\sum_{j=1}^n \frac{a_{j,n}}{p_j-i\xi} \right)
\end{eqnarray*}
where we used Fubini theorem for the third equality and condition (\ref{condition-1}) to justify the last equality.  The last term is now evaluated by using the obvious decomposition
\begin{equation*}
\prod_{j=1}^{n}\frac{x-\lambda_j}{x+\lambda_j+1}=1-\sum_{j=1}^{n}
\frac{a_{j,n}}{x+\lambda_j+1},\quad x\neq -\lambda_j-1, \quad
j=1,2,\cdots, n.
\end{equation*}
Note that the latter decomposition allows as well to resolve the system (\ref{system-0}).
\end{proof}

Our aim now is to look for the analogue of Proposition \ref{moving-average-representation} when $n=+\infty$. Observe that for the transform (\ref{MA}) to be well defined we merely need $\eta\in L^2(\mathbb{R}_+)$ and we can even take $\eta\in L^2_{\mathbb{C}}(\mathbb{R}_+)$. Of course, we need then to work with the Fourier-Plancherel transform instead of the Fourier transform. We recall that this is connected to the Hardy class $H^2_+$ of holomorphic functions $H$ in the upper half-plane $\mathbb{C}_+=\{z\in \mathbb{C}, \hbox{Im}(z)>0\}$ such that $\sup_{b>0}\int_{\mathbb{R}}\! |H(a+ib)|\,da<\infty$, see \cite{DM}.
 We gather in the following result some well known results which are mostly taken from \cite{hhm-00} and \cite{jeulin-yor93}; for completeness a full proof will be given.
\begin{Theorem}\label{Infinite-order-Kernel} Assuming that  $W$  is a Brownian motion indexed by $\mathbb{R}$
and $B$ is a standard Brownian motion satisfying (\ref{def-W}) and (\ref{def-B}) then the following assertions are equivalent.
\begin{itemize}
\item[(1)] $\Lambda$  satisfies (\ref{MS}).
\item[(2)] There exists a transform $S$ of the form (\ref{MA}) associated to $\Lambda$ such that $(S(W)_u, u\in \mathbb{R})$ is an Ornstein-Uhlenbeck process and
\begin{equation}\label{Loss-inf-0}
 \mathcal{H}_u(W)=\mathcal{H}_u(S(W))\oplus \hbox{Span}\{ \int_{-\infty}^{u}\! e^{p_j u}\,  dW_u, j=1,2, \cdots\}
 \end{equation}
 for all $u\in \mathbb{R}$  where   $\mathcal{H}(S(W))_u$ stand for the closed linear span of $\{S(W)_r, r\leq u\}$.
 \item[(3)] There exists a transform $T$ of the form (\ref{Volterra-general})  such that $(T(B)_t, t\geq 0)$ is a standard Brownian motion and
\begin{equation}\label{Loss-inf}
H_t(B)=H_t(T(B))\oplus \hbox{Span} \{\int_0^t\! s^{\lambda_j}\, dB_s, j=1,2, \cdots \}
\end{equation}
 for all $t>0$.
\end{itemize}
\end{Theorem}

\begin{proof}[Proof of Theorem \ref{Infinite-order-Kernel}]

$(1)\Leftrightarrow (2)$ Assuming that equation (\ref{MS}) is not satisfied then by M\"untz-Szasz theorem, see e.g.$\,$\cite{Borwein-Erdelyi-95}, the sequence $(f_k)$ is complete in $L^2[0,t]$ for all $t>0$. It follows that the sequence $\left( e^{p_j x}, j=1,2, \cdots \right)$ is total in $L^2(-\infty, a]$ for all $a$ real.
 Hence $ H_u(W)=\hbox{Span}\{ \int_{-\infty}^{u}\! e^{p_j s}\, dW_s, j=1,2, \cdots\}$
  which shows that it is not possible to construct a transform $S$ satisfying (\ref{Loss-inf-0}) such that $S(W)$ is an Ornstein-Uhlenbeck process. Conversely,  condition (\ref{MS})  ensures the convergence of the infinite product (\ref{Infinite-product-Fourier}).
   It is seen in Theorem 2 of \cite{hhm-00}, see also (\cite{jeulin-yor93}, p.$\,$60), that under the condition (\ref{MS}) the function $H:\mathbb{R}\rightarrow \mathbb{C}$ defined by
\begin{equation}\label{Infinite-product-Fourier}
H(\xi)=\frac{1}{1/2-i\xi}\prod_{j=1}^{\infty}
\frac{\xi-ip_j}{\xi+ip_j}\frac{|1-p_j|}{1-p_j}
\end{equation}
is the Fourier-Plancherel transform of a function $\eta_{\infty}\in L^2_{\mathbb{C}}(\mathbb{R}_+, dx)$, where $dx$ is the Lebesgue measure; this follows from the fact that $H\in H^2_+$.
  Note that $\eta_{\infty}$ is real-valued since  $\overline{H(\xi)}=H(-\overline{\xi})$ for $\xi\in \mathbb{R}$. Let $S_{\infty}$ be defined by (\ref{MA})  with $\eta=\eta_{\infty}$. The process $(S_{\infty}(W)_u, u\in \mathbb{R})$ is a continuous stationary Gaussian process with spectral measure
$(2\pi)^{-1}|H(\xi)|^2=(2\pi)^{-1}(\xi^2+1/4)^{-1}$ and covariance function
\begin{eqnarray*}
\mathbb{E}\left[ S_{\infty}(W)_{u}S_{\infty}(W)_{v}\right]&=&\int_{-\infty}^{u \wedge v}\! \eta_{\infty}(u-r)\eta_{\infty}(v-r)\, dr\\
&=&\int_0^{\infty}\! \eta_{\infty}(r)\eta_{\infty}(u\vee v-u\wedge v-r)\, dr\\
&=&\frac{1}{2\pi}\int_{-\infty}^{+\infty}\! |\hat{\eta}_{\infty}(\xi)|^2 e^{i(u\vee v-u\wedge v)\xi}\, d\xi\\
&=&\frac{2}{\pi}\int_{-\infty}^{+\infty}\! e^{i(u\vee v-u\wedge v)\xi}\frac{d\xi}{1+4\xi^2}\\
&=&e^{-\frac12 |u-v|}
\end{eqnarray*}
for $u$ and $v\in \mathbb{R}$. Thus, $S_{\infty}(W)$ is a stationary Ornstein-Uhlenbeck process. Now, for all $u<v$ and $j=1$, $2$, $\cdots$, $S_{\infty}(W)_u$ is independent of $\int_{-\infty}^{v}\! e^{p_j r}\, dW_r$ since
\begin{eqnarray*}
\mathbb{E}\left[S_{\infty}(W)_u \int_{-\infty}^{v}\! e^{p_jr}\, dW_r \right]&=& \int_{-\infty}^{u}\! e^{(\lambda_j+1/2) r} \eta_{\infty}(u-r)\, dr\\
&=&e^{p_j u}\int_0^{\infty}\! e^{-p_js}\eta_{\infty}(s)\, ds=0
\end{eqnarray*}
where the last equality is obtained from the fact $ip_j$ is a zero point of $H$.

\noindent $(2)\Leftrightarrow (3)$   Assuming (2), we can set $\rho_{\infty}=U^{-1}\eta_{\infty}$, where $\eta_{\infty}$ is as above, and define $B$ by (\ref{def-B}). Since $\eta_{\infty}\in L^2(\mathbb{R}_+)$,
 we clearly have  that $\rho_{\infty}\in \mathcal{M}$. Let us now define $T_{\infty}$ by (\ref{Volterra-general}) where $\rho_{\infty}$ and $B$ are as prescribed above.  Clearly, $T_{\infty}(B)$ is a standard Brownian motion. Furthermore, we have
 \[H_t(T_{\infty}(B))\perp \int_0^t\! u^q\, dB_u
 \]
 for all $t>0$, with $q>-1/2$, if and only if
 \[\mathcal{H}(S_{\infty}(W))_u\perp \int_{-\infty}^u\! e^{pr}\, dW_r
 \]
 for all $u\in \mathbb{R}$,
  with $p=q+1/2$. Conversely, by reversing the steps we see that $(3)$ implies $(2)$.
  \end{proof}
\begin{Remark} When we outlined the connection with stationary processes, we could have considered  $W^{(\alpha)}$ satisfying $\int_0^{\infty}\! \varphi(s)\, dB_s=\int_{\mathbb{R}}\! V^{(\alpha)}\varphi (r)\, dW_r^{(\alpha)}$ for $\varphi\in L^2(\mathbb{R}_+)$, for some $\alpha>0$, where the isometry $V^{(\alpha)}:L^2(\mathbb{R}_+)\rightarrow L^2(\mathbb{R})$  is defined by $V\varphi(u)=\sqrt{\alpha}e^{\alpha u/2}\varphi(e^{\alpha  u})$, i.e. $dW^{(\alpha)}=\alpha^{-1/2} e^{-\alpha u/2}dB(e^{\alpha u})$ with $W^{(\alpha)}_0=B_1$. But, we need to use $U^{(\alpha)}(\phi)(u)=\alpha^{-1/2} e^{\alpha u/2} \phi(e^{\alpha u})$ instead of $U$. The authors used this transformation with $\alpha=2$ in \cite{hhm-00} and \cite{hm-2004}. Of course, the conclusions are the same up to working with $p_j=2\lambda_j+1$ instead of $p_j=\lambda_j+1/2$.
\end{Remark}
\begin{Theorem} \label{ifinite-semimartingale} There exists a transform of the form (\ref{Volterra-general}) such that
 $(T_{\infty}(B)_t, t\geq 0)$ is a Brownian motion satisfying (\ref{Loss-inf}) and is a semimartingale in $\mathcal{F}^B$ if and only if $(\lambda_k)$ is bounded and satisfies the M\"untz-Szasz condition (\ref{MS}) i.e.$\,$$\sum_{j=1}^{\infty}p_j<\infty$.
\end{Theorem}

\begin{proof}[Proof of Theorem \ref{ifinite-semimartingale}]
 By Theorem \ref{Infinite-order-Kernel} there exists a  transform of the form (\ref{Volterra-general}) such that (\ref{Loss-inf}) holds and $(T(B), t\geq 0)$ is a Brownian motion if and only if  condition (\ref{MS}) is satisfied. Assuming that (\ref{MS}) is satisfied, let us check that the semimartingale property cannot hold  if there exists a subsequence $(n_k)_{k\in \mathbb{N}}$ such that  $\lambda_{n_k} \rightarrow \infty$. For this, let us quote an argument, from \cite{Erraoui-Ouknine-2008} and \cite{hm-2004}, to show that we necessarily have  $\int_1^{\infty}\! \rho'(u)^2\,  du=+\infty$ in this case. Using the fact that $T(B)_1$ is independent of $\int_0^1\! u^{\lambda_{n_k}}\, dB_u$ and a change of variables, we obtain  $\int_0^1\! u^{\lambda_{n_k}} \rho(1/u)\, du=\int_1^{\infty}\! u^{-(\lambda_{n_k}+2)}\rho(u)\, du=0$. This, when combined with integration by parts, yields
\begin{eqnarray*}
\int_1^{\infty}\! u^{-(\lambda_{n_k}+1)}\rho'(u)\, du&=&\int_1^{\infty}\! u^{-(\lambda_{n_k}+1)}\rho'(u)\, du-(1+\lambda_{n_k})\int_1^{\infty}\! u^{-(\lambda_{n_k}+2)}\rho(u)\, du\\
&=& \left[\rho(u) u^{-(\lambda_{n_k}+1)} \right]_{1}^{\infty}=-\rho(1).
\end{eqnarray*}
  By using the Cauchy-Schwartz inequality, we obtain $(1+2\lambda_{n_k})|\rho(1)|^2 \leq \int_1^{\infty}\! \rho'(u)^2\,  du$ which implies that $\int_1^{\infty}\! \rho'(u)^2 du=\int_0^1\! \rho'(1/v)^2 v^{-2}\, dv=+\infty$. If $T(B)$ were an $\mathcal{F}_t^{B}$ semimartingale then, by applying Theorem 6.5 of \cite{Knight}, we would have the existence of $g$ such that $\rho(t)=c+\int_1^{.}\! y^{-1}g(y)\, dy$ with $g\in \mathcal{M}$ and $c\neq 0$. But then  $\rho'(y)=g(y)/y$, $y>1$,  and we should have $\int_0^1\! g^2(1/v)\, dv=\int_0^1\! \rho'(1/v)^2v^{-2}\, dv<\infty$. This contradicts the fact that $\int_1^{\infty}\! \rho'(v)^2\, dv=+\infty$.

  Let us now examine the case where $(\lambda_k)$ satisfies (\ref{MS}) and is bounded. i.e. $\sum_1^{\infty} (1+2\lambda_j)<\infty$. We shall first show that $K_n$ converges as $n\rightarrow \infty$ in $L^2[0,1]$. Using Proposition \ref{properties}, for any positive integers $n>m$, we can write
  \begin{eqnarray*}
  \int_0^1\! \left(K_n(u)-K_m(u) \right)^2\, du&=&\int_0^1\! \left(\sum_{j=m+1}^n(1+2\lambda_j)L_j(u) \right)^2\, du\\
  &=&\sum_{j=m}^n(1+2\lambda_j)\sum_{k=m+1}^n(1+2\lambda_k)\int_0^1\! L_j(u)L_k(u)\, du\\
  &=&\sum_{j=m+1}^n(1+2\lambda_j)\rightarrow 0, \quad \hbox{as} \; m,n\rightarrow \infty,
  \end{eqnarray*}
where we have used the fact that $L_1$, $L_2$, $\cdots$, are orthogonal and $\int_0^1\! L_j^2(r)\, dr=1/(1+2\lambda_j)$ for $j=1, \cdots, n$.
  This shows that $(K_n)$ is a Cauchy sequence in $L^2[0,1]$. Hence, it must converge to a limit which we denote by $K$.  With $\rho_n(.)=1-\int_1^{.}\! K_n(1/r) r^{-1}\, dr $, $n=1,2 \cdots$, let us show that
  $\rho_n(1/.) \rightarrow \rho(1/.)$ in $L^2[0,1]$ where
  \begin{equation}\label{expression-rho}\rho=1-\int_1^{\infty}\! K(1/r)r^{-1}\, dr.
   \end{equation}
   To this end, we quote from \cite{CDM-Y-89}  the following variant of Hardy inequality. For any $g\in L^2[0,1]$, we have
  \[
  \int_0^1\! \left(\int_u^1\! g(r)r^{-1}dr \right)^2\, du\leq 4\int_0^1\! g^2(u)\, du.
  \]
  Now, we can write
  \begin{eqnarray*}
  \int_{0}^1\! \left(\rho_n(1/v)-\rho(1/v) \right)^2\, dv&=&\int_{0}^1\! \left( \int_{1}^{1/v}\! (K_n(1/z)-K(1/z))z^{-1}\, dz \right)^2\, dv\\
  &=& \int_0^1\! \left(\int_v^{1}\! (K_n(r)-K(r)) r^{-1}\, dr \right)^2\, dv\\
  &\leq &  4\int_0^{1}\! (K_n(u)-K(u))^2\,  du  \rightarrow 0\quad \hbox{a} \; n\rightarrow \infty,
  \end{eqnarray*}
  which is our claim. It follows that
  \[
  T_n(B)_t=\int_0^t\! \rho_n(t/s)\, dB_s\rightarrow \int_0^t\! \rho(t/s)\,dB_s=:T(B)_t \quad \hbox{in}\; L^2(\Omega, \mathcal{F}, \mathbb{P}).
  \]
   Similar arguments show that $\eta_n\rightarrow \eta$ in $L^2(\mathbb{R}_+)$, with $\eta(t)={\bf 1}_{t>0}U\circ \rho (t)$, and
$S_n(W)_u\rightarrow \int_{-\infty}^u\! \eta(u-v)\,dW_v=:S(W)_u $ in  $L^2(\Omega, \mathcal{F}, \mathbb{P})$. We need to show that $\eta=\eta_{\infty}$ where $\eta_{\infty}$ is  the Fourier inverse of the function $H$ defined by (\ref{Infinite-product-Fourier}). Applying Plancherel Theorem, with $\hat{\eta}_n$ prescribed by (\ref{product}), we see that $\hat{\eta}_n \rightarrow \hat{\eta}$ in $L^2_{\mathbb{C}}(\mathbb{R})$. But by Theorem 13.12 in \cite{mashreghi}, we know that
 \[
 \Pi_n(\xi) \rightarrow \Pi_{\infty}(\xi):= \prod_{j=1}^{\infty}
\frac{\xi-ip_j}{\xi+ip_j} \quad \hbox{as}\quad n\rightarrow \infty,
\]
 uniformly on compact subsets of $\mathbb{R}\setminus\{0\}$. It follows that  $\hat{\eta}_n\rightarrow H$, as $n\rightarrow \infty$, uniformly on compact subsets of $\mathbb{R}\setminus\{0\}$. We conclude that necessarily $\hat{\eta}=H$ a.e. which implies that $\eta=\eta_{\infty}$ a.e.. It follows from the proof of Theorem \ref{Infinite-order-Kernel} that $S_{\infty}(W)$ is an Ornstein-Uhlenbeck process which implies that $T_{\infty}(B)$ is a standard Brownian motion. Due to the fact that  $\rho$ is given by (\ref{expression-rho}), Proposition 15 on p.$\,$69 of \cite{jeulin-yor93} shows that $T(B)$ is a semimartingale with respect to $\mathcal{F}^{B}$.
\end{proof}

\begin{Definition} A M\"untz transform is a transform $T$ of the form (\ref{Volterra-general}) such that  $(T(B)_t, t\geq 0)$ is a standard Brownian motion for any Brownian motion $(B_t, t\geq 0)$ and
  the orthogonal decomposition
$H_t(B)=H_t(T(B))\oplus G_t(\lambda_1,\cdots, \lambda_n; B)$, for some sequence of reals $-1/2<\lambda_1, \cdots, \lambda_n$ and  $n\in\{ 1, 2, \cdots\}$, holds for all $t>0$. We call $n$ the order of the transform.  The corresponding kernel $\rho(./.)$ (or $k_n$ in the semimartingale case)  is called a M\"untz kernel of order $n$.
\end{Definition}

\begin{Remark}
As a by-product of the discussion for the order to be infinite we
mention the following result. Let $\varphi$ be a $C^{\infty}([0,1])$
function satisfying $|\varphi^{(m)}| \le M$  for all $m$, where $M$ is some positive constant.
Then $\varphi$ is a solution to the integral equation $\varphi(u) =
\int_0^1\! \varphi(uv) \varphi(v)\, dv$, defined on $[0,1]$, if and only if
$\varphi(.)=k_n(1,.)=K_n(.)$ where $\lambda_j=j$ for $j\geq 0$ and $n$ is
some finite positive integer.
\end{Remark}
\begin{Remark}\label{Ergodic-thm} For $n\in \{1, 2, \cdots \}$,
$k_n$ and $T_n$ as above, introduce the notations
$T^{(0)}_n=Id$, $T^{(1)}_n=T_n$ and $T^{(m)}_n =
T^{(m-1)}_n \circ T_n$, for $m \geq 2$, where $\circ$
stands for the composition rule for the iterated transforms. We clearly have for $m$ positive integer
\[
\cdots {\mathcal F}_t^{T_n^{(m+1)}(B)}\varsubsetneq {\mathcal F}_t^{T_n^{(m)}(B)}\varsubsetneq \cdots {\mathcal F}_t^{T_n}(B)\varsubsetneq {\mathcal F}_t(B).
\]
Furthermore, since
we are in the homogeneous case,  we can show that the
decomposition
\[
{\mathcal F}_t^{B} = \bigotimes_{k = 1}^{\infty}
\sigma \left( \int_0^t\! u^{\lambda_j} \, dT^{(k)}_n(B)_u, 1\leq
j\leq n\right)
\]
 holds true. Here,
by $\mathcal{F}\otimes \mathcal{G}$, for two $\sigma$-algebras $\mathcal{F}$ and $\mathcal{G}$,  we mean $\mathcal{F}\vee
\mathcal{G}$ with independence between $\mathcal{F}$ and
$\mathcal{G}$. It follows that  M\"untz transforms
are strongly mixing and ergodic. We also refer to
\cite{jeulin-yor93} for a proof of this, in a more general framework, which uses the connection to stationarity.
\end{Remark}
\begin{example}
For $r>0$, let us set $\lambda_j=(j^{-r}-1)/2$ and so $p_j=j^{-r}/2$, $j=1,2,\cdots$.  For $n$ positive integer, we obtain
\begin{equation*}
a_{k,n}=\frac{2}{ k^{r}}\prod_{j=1, j\neq k}^n
\frac{j^r+k^r}{j^r-k^r}, \quad k=1,2,\cdots n.
\end{equation*}
 The hyperharmonic series  $\sum_1^{\infty}(1+2\lambda_j)=\sum_{1}^{\infty} j^{-r}$ converges  if and only if $r>1$ which, by Theorem \ref{ifinite-semimartingale}, is the necessary and sufficient condition for  the existence of an associated M\"untz transform of infinite order. If $r>1$ then
\begin{eqnarray*}
H(\xi)= \frac{1}{1/2-i \xi}\prod_{k=1}^{\infty}
\frac{k^r-i/(2\xi)}{k^r+i/(2\xi)}, \quad \xi\in \mathbb{R}.
\end{eqnarray*}
If furthermore $r$ is a an integer then
\begin{equation*}
H(\xi)=-
 (1/2-i\xi)^{-1} \prod_{j=1}^{2r} \Gamma(-(i/(2\xi))^{1/r}
\omega_{2r} ^j)^{(-1)^{j+1}}
\end{equation*}
 where we have used the relationship
 \begin{equation*}
\prod_{j \geq 1} \frac{j^r-z^r}{j^r+z^r} =
-\prod_{j=1}^{2r}\Gamma(-z\omega_{2r} ^j)^{(-1)^{j+1}},
\end{equation*}
 with $\omega_{2r}=\exp(\pi i/r)$, which is valid for $z\notin \{0, 1, 2, \cdots\}$ and is found in (\cite{BBR-04}, pp.$\,$6-7). Since the residue of $\Gamma(z)$ at $z=-k$ is $(-1)^{k}/k!$, we have
 \begin{equation*}
 (k^r-z^r)\Gamma(-z)\rightarrow \frac{(-1)^k}{k!}rk^{r-1}\quad \hbox{as} \quad z\rightarrow k.
 \end{equation*}
 It follows that
\begin{equation*}
\prod_{j=1, j \neq k} \frac{j^r-k^r}{j^r+k^r} =
(-1)^{k+1}\frac{2k(k!)}{r}\prod_{j=1}^{2r-1}\Gamma(-k\omega_{2r}
^j)^{(-1)^{j+1}}
\end{equation*}
which leads to
\begin{eqnarray*}
a_{k,n}&\rightarrow& \frac{2}{k^r} \left\{
(-1)^{k+1}\frac{2k(k!)}{r}\prod_{j=1}^{2r-1}\Gamma(-k\omega_{2r}
^j)^{(-1)^{j+1}} \right\}^{-1}\\
&=&
(-1)^{k+1}\frac{r}{k^{r+1}k!}\prod_{j=1}^{2r-1}\Gamma(-k\omega_{2r}
^j)^{(-1)^{j}}
\end{eqnarray*}
as $n\rightarrow \infty$.
\end{example}

\noindent{\bf Acknowledgment:} The authors would like to thank Y. Hibino for
pointing out a mistake in a previous version of Theorem \ref{moving-average-representation}. We are grateful to the anonymous referee for a careful reading of the manuscript and for providing interesting comments which lead to an improvement of this paper.
The first author is indebted to l'Agence Nationale de la Recherche  for  the research grant ANR-09-Blan-0084-01.

\vspace{5mm} \noindent{\footnotesize $^{(1)}$ Department of
Statistics, University of Warwick, {\sc CV4 7AL, Coventry}. L.alili@warwick.ac.uk
\\
 $^{(2)}$Department of Mathematics,
National Taitung University, {\sc no. 369, Sec. 2, Shikang Rd,
Taitung, Taiwan}. ctwu@nttu.edu.tw
}



\begin{thebibliography}{99}


\bibitem{alili00} Alili, L.: Canonical decomposition of certain generalized Brownian bridges. \emph{Electron. Comm. Probab.}, \textbf{7},  (2002), 27--36.

\bibitem{alili-wu-01} Alili, L. and Wu, C.-T.: Further results on some singular linear stochastic differential equations. \emph{Stochastic Process. Appl.}, \textbf{119},   no. 4, (2009), 1386--1399.

\bibitem{Bass-2008} Basse, A.: Gaussian moving averages and semimartingales. \emph{Electron. J. Probab.}, \textbf{ 13}, no. 39, (2008), 1140--1165.

\bibitem{Baudoin} Baudoin, F.: Conditioned stochastic differential equations: Theory, Examples and Application to Finance. \emph{Stochastic Process. Appl.}, \textbf{100}, (2002), 109--145.

\bibitem{BBR-04} Borwein, J.M.,  Bailey, D.H. and Girgensohn, R.: Experimentation in mathematics. Computational paths to discovery. \emph{A K Peters, Ltd., Natick, MA,} x+357 pp., 2004.

\bibitem{Borwein-Erdelyi-95} Borwein, P. and  Erd\'elyi, T.: Polynomials and polynomial inequalities. \emph{Graduate Texts in Mathematics}, \textbf{161}, Springer-Verlag, New York, x+480 pp., 1995.

\bibitem{Borowein-Erdelyi-Zhang-94} Borwein, P.,  Erd\'elyi, T. and  Zhang, J.: M\"untz systems and orthogonal M\"untz-Legendre polynomials. \emph{Trans. Amer. Math. Soc}, \textbf{342}, no.2, (1994), 523--542.

\bibitem{CM_Jeulin} Chaleyat-Maurel, M. and Jeulin, T.: Grossissement gaussien de la filtration brownienne. \emph{C. R. Acad. Sci. Paris S\'er. I Math.}, \textbf{196}, no 15, (1983), 699--702.

\bibitem{Cheridito-2004} Cheridito, P.: Gaussian moving averages, semimartingales and option pricing. \textit{Stochastic Process. Appl.}, \textbf{ 109}, no. 1, (2004), 47--68.

\bibitem{Chiu-1995} Chiu, Y.: From an example of L\'evy's. \emph{S\'eminaire de Probabilit\'es}, XXIX, Lecture Notes in Math., Springer,   \textbf{1613}, Springer, Berlin,  (1995),  162--165.

\bibitem{CDM-Y-89} Donati-Martin, C. and  Yor, M.: Mouvement brownien et in\'egalit\'e de Hardy dans $L^2$. \emph{S\'eminaire de Probabilit\'es}, XXIII, Lecture Notes in Math., Springer, Berlin, \textbf{1372}, (1989), 315-323.

\bibitem {DM} Dym, H. and McKean, H.P.:  Gaussian processes, fluctuation theory, and the inverse spectral problem. \emph{Academic Press},  New York-London, xi+335 pp., 1976.

\bibitem{Duren} Duren,  Peter L.: Theory of $H^p$ spaces. Pure and Applied Mathematics, Vol. 38, \emph{Academic Press}, New York-London, xii+258 pp. 1970.

\bibitem{fwy99} F\"ollmer, H., Wu, C-T. and  Yor, M.: On weak Brownian motions of arbitrary order.
         \emph{Ann. Inst. H. Poincar\'e Probab. Statist.}, \textbf{36}, no. 4, (2000), 447--487.

\bibitem{Eraoui-Essaky} Erraoui, M. and  Essaky E.H.:  Canonical representation for Gaussian processes.  \emph{S\'eminaire de Probabilit\'es}, XLII, Lecture Notes in Math.,  Springer, Berlin, \textbf{1979}, (2009),  365--381.

\bibitem{Erraoui-Ouknine} Erraoui, M. and  Ouknine, Y.: Noncanonical representation with an infinite-dimensional orthogonal complement. \emph{Statist. Probab. Lett.}, \textbf{78},  no. 10, (2008), 1200--1205.

\bibitem{Erraoui-Ouknine-2008} Erraoui, M. and Ouknine, Y.: Equivalence of Volterra processes: degenerate case.  \emph{Statist. Probab. Lett.},\textbf{ 78},  no. 4,  (2008),   435--444.

\bibitem{Gohberg-Koltracht} Gohberg, I. and Koltracht, I.: Triangular factors of Cauchy and Vandermonde matrices. \emph{Integr. Equat. Oper. Th.}, \textbf{26}, (1996), 46--59.

\bibitem{Hibino-96} Hibino, Y.: Construction of noncanonical representations of Gaussian process.
          \emph{J. of the Fac. of Lib. Arts}, Saga University, no. 28, (1996),  1-7.

\bibitem{hhm} Hibino,Y., Hitsuda, M. and Muraoka, H.: Construction of noncanonical representations of a
         Brownian motion. \emph{Hiroshima Math. J.}, \textbf{27}, no. 3, (1997),  439--448.

\bibitem{hhm-00} Hibino, Y., Hitsuda, M. and Muraoka, H.: Remarks on a noncanonical representation for a stationary
         Gaussian process. \emph{Acta Appl. Math.}, \textbf{63}, no. 1-3,(2000), 137--139.

\bibitem{hm-2004} Hibino, Y. and Muraoka, H.: Volterra representations of Gaussian processes with an infinite-dimensional orthogonal complement.
         \emph{Quantum Probability and Infinite Dimensional Analysis.} From Foundations to Applications, Eds. M. Sh\"urmann and
          U. Franz, QP-PQ: Quantum Probability and White Noise Analysis, \textbf{18},
          World Scientiffic, (2005), 293--302.

\bibitem{Hida-60} Hida, T.: Canonical representations of Gaussian processes and their applications. \emph{Mem. Coll. Sci. Univ. Kyoto Ser. A. Math.}, \textbf{33}, (1960), 109--155.

\bibitem{hida93} Hida, T. and  Hitsuda, M.: \emph{Gaussian processes}. Translated from the 1976 Japanese original by the authors. Translations of Mathematical Monographs, \textbf{120}. American Mathematical Society, xvi+183 pp. 1993.

\bibitem{Hitsuda-68} Hitsuda, M.: Representations of Gaussian processes equivalent to Wiener process. \emph{Osaka J. Math.}, \textbf{5}, (1968),  299--312.

\bibitem{jeulin80} Jeulin, Thierry: Semi-martingales et grossissement d'une filtration. \emph{Lecture Notes in Mathathematics}, \textbf{833}, Springer, Berlin, ix+142 pp. 1980.

\bibitem{jeulinyor85} Jeulin, Th. and  Yor, M.: Grossissements de filtrations: exemples et applications. \emph{Lecture Notes in Mathematics,} \textbf{1118},
          Springer-Verlag, Berlin, vi+315 pp. 1985.

\bibitem{jeulin-Yor-90} Jeulin, T. and Yor, M.: Filtration des ponts browniens et \'equations differentielles lin\'eaires. \emph{S\'eminaire de Probabilit\'es}, XXIV, Lecture Notes in Math., \textbf{1426}, Springer, Berlin, (1990), 227--265.

\bibitem{jeulin-yor93} Jeulin, T. and Yor, M.: Moyennes mobiles et semimartingales. \emph{S\'eminaire de Probabilit\'es}, XXVII, Lecture Notes in Math., \textbf{1557}, Springer, Berlin, (1993), 53--77.

\bibitem{Kara-50} Karhunen, K.: \"Uber die Struktur stati\"onarer zuf\"alliger Functionen. \emph{Ark. Mat.}, \textbf{1}, (1950), 141--160.

\bibitem{Knight} Knight,  Frank B.: Foundations of the predictible process. Oxford studies in Probability, vol. 1, \emph{Clarendon Press}, Oxford.  xii+248 pp. 1992.

\bibitem{Levy-book} L\'evy, P.: Processus stochastiques et mouvement brownien. \emph{Gauthier-Villars $\&$ Cie}, Paris, vi+438 pp. 1965.

\bibitem{Levy-56} L\'evy, P.: Sur une classe de courbes de l'espace de Hilbert et sur une \'equation int\'egrale non lin\'eaire. \emph{ Ann. Sci. Ecole Norm. Sup.}, \textbf{73}, (1956), 121--156.

\bibitem{Levy-57} L\'evy, P.: Fonctions al\'eatoires \`a corr\'elation lin\'eaire. \emph{ Illinois J. Math.}, \textbf{1}, (1957),  217--258.

\bibitem{mashreghi} Mashreghi, J.:  Representation theorems in Hardy spaces.  London Mathematical Society Student Texts, \textbf{74}. \emph{Cambridge University Press}, Cambridge, xii+372 pp. 2009.

\bibitem{Schechter} Schechter, S.: On the inversion of certain matrices. \emph{Math. Tables Aids Comp.}, \textbf{13}, (1959), 73--77.

\bibitem{yor92} Yor, M.: Some aspects of Brownian motion, Part I: Some special functionals. \emph{Lectures in Mathematics ETH Z\"urich}. Birkh\"auser Verlag, Basel, x+136 pp. 1992.


\end{thebibliography}
\end{document}